%% file: AmarAndreucciTimofte_preprint.tex
[1994/06/01]

\documentclass[12pt,reqno,intlimits,a4paper,english]{amsart}

\usepackage{graphicx,pstricks,pst-3dplot}
\usepackage{pst-plot}
\usepackage{calc}
\usepackage{babel}
\usepackage{amssymb}

\numberwithin{equation}{section}
\newcounter{hours}\newcounter{minutes}
\newcommand{\hourandminute}{\setcounter{hours}{\time/60}%
\setcounter{minutes}{\time-\value{hours}*60}%
\thehours.\theminutes}
\newcommand{\Versione}{\jobname\ \today\ \hourandminute}
\newlength{\Indent}
\newlength{\Parskip}
\theoremstyle{plain}
\newtheorem{thm}{Theorem}[section]     
\newtheorem{prop}[thm]{Proposition}

\theoremstyle{remark}

\newtheorem{Remark}[thm]{Remark}
\newenvironment{remark}{\begin{Remark}}{\qed\end{Remark}}

\theoremstyle{definition}
\newtheorem{Defin}[thm]{Definition}

\DeclareMathOperator{\Div}{div}

\newcommand{\Iion}{I_{\textrm{ion}}}
\newcommand{\aaa}{a}
\newcommand{\bbb}{b}

\newcommand{\RN}{\Bbb{R}^{N}}
\newcommand{\R}{\Bbb{R}}

\newcommand{\di}{\,\text{\rmfamily\upshape d}}

\newcommand{\Om}{\varOmega}

\newcommand{\CC}{\mathcal{C}}

\def\X{{\mathcal X}}
\def\W{{\mathcal W}}

\newcommand{\dfint}{\sigma_{1}^B}

\newcommand{\dfout}{\sigma_2^B}

\newcommand{\dfdis}{\sigma^{D}}

\newcommand{\capuno}{\alpha}
\newcommand{\capdue}{\beta}
\newcommand{\ubuno}{u^{B}_1}
\newcommand{\ubdue}{u^{B}_2}
\newcommand{\ud}{u^{D}}

\newcommand{\dfboth}{\sigma}

\newcommand{\Omint}{\Om^{D}}
\newcommand{\Omout}{\Om^{B}}
\newcommand{\Memb}{\varGamma}

\newcommand{\Permemb}{\varGamma}

\newcommand{\const}{\gamma}

\newcommand{\testb}{\varphi_B}
\newcommand{\testd}{\varphi_D}
\newcommand{\testduno}{\varphi_D^1}
\newcommand{\testddue}{\varphi_D^2}
\newcommand{\win}{{\textsc{w}_o}}
\newcommand{\upsilonuno}{\upsilon^B_1}
\newcommand{\upsilondue}{\upsilon^B_2}
\newcommand{\upsilond}{\upsilon^D}
\newcommand{\overlineubuno}{\overline u^{B}_1}
\newcommand{\overlineubdue}{\overline u^{B}_2}
\newcommand{\overlineud}{\overline u^{D}}

\makeindex
\begin{document}

\title
{Well-posedness for a modified bidomain model describing bioelectric activity
in damaged heart tissues}
\author{M. Amar$^\dag$ -- D. Andreucci$^\dag$ -- C. Timofte$^\S$\\
\hfill \\
$^\dag$Dipartimento di Scienze di Base e Applicate per l'Ingegneria\\
Sapienza - Universit\`a di Roma\\
Via A. Scarpa 16, 00161 Roma, Italy
\\ \\
$^\S$University of Bucharest\\
Faculty of Physics\\
P.O. Box MG-11, Bucharest, Romania
}

\begin{abstract}
We prove the existence and the uniqueness of a solution for a modified bidomain model, describing the electrical behaviour of the cardiac
tissue in pathological situations. The leading idea is to reduce the problem to an
abstract parabolic setting, which requires to introduce several auxiliary differential systems
and a non-standard bilinear form.
The main difficulties are due to the degeneracy of the bidomain system and to its non-standard coupling with a diffusion equation,
accounting for the presence of the pathological zone in the heart tissue.
\medskip


  \textsc{AMS-MSC:} 35K90, 35A01, 35K20, 35Q92
  \bigskip

  \textbf{Acknowledgments}: The first author is member of the \emph{Gruppo Nazionale per l'Analisi Matematica, la Probabilit\`{a} e le loro Applicazioni} (GNAMPA) of the \emph{Istituto Nazionale di Alta Matematica} (INdAM).
The second author is member of the \emph{Gruppo Nazionale per la Fisica Matematica} (GNFM) of the \emph{Istituto Nazionale di Alta Matematica} (INdAM). The last author wishes to thank \emph{Dipartimento di Scienze di Base e Applicate per l'Ingegneria} for the warm hospitality and \emph{Universit\`{a} ``La Sapienza" of Rome} for the financial support.
\end{abstract}

\maketitle

\section{Introduction}\label{s:introduction}
In this paper, we are interested in studying a modified version of the famous {\it bidomain model} (see, e.g.,
\cite{CBetal,KS,NK} and the references therein; see, also, the references quoted in \cite[Introduction]{Collin:Imperiale:2018}),
which is one the most well-known mathematical models in cardiac electrophysiology. This is a topic of major interest in biomedical research.

In the classical bidomain model, at a macroscopic scale, the electric activity of the heart is governed by a system of two degenerate reaction-diffusion partial differential equations for the averaged intra-cellular and, respectively, extra-cellular electric potentials, along with the transmembrane potential, coupled in a nonlinear manner to ordinary differential equations describing the dynamics of the ion channels.
 The well-posedness of the bidomain model has been studied, for different nonlinear ionic models and by using different techniques, by several authors (see, for instance, \cite{BK,Bourgault:Coudiere:Pierre:2009,Collin:Imperiale:2018,Davidovic:2016,Grandelius:2019,KLPW,Nagumo:Arimoto:Yoshizawa:1962,
 Pennacchio:Savare:Franzone:2005,Veneroni:2009}).

The bidomain model is suitable for describing the propagation of the action potential in a perfectly healthy cardiac tissue, but it is no longer valid (even if one tries to {\it ad-hoc} modify some of its relevant modeling parameters) in pathological situations.
Models, taking into account the presence in the cardiac tissue of damaged zones, called {\it diffusive inclusions} and assumed to be passive electrical conductors, were proposed in \cite{Boulakia:Fernandez:Gerbeau:Zemzemi:2008,
CDP1,CDP2,Davidovic:2016}.

From a mathematical point of view, such models consist in a bidomain system coupled with a diffusion equation. More precisely,
one has a degenerate reaction-diffusion
system of partial differential equations modeling the intra-cellular and, respectively, the extra-cellular electric potentials of the healthy cardiac tissue, coupled with an elliptic equation for the passive regions and with an ordinary differential equation describing the cellular membrane dynamics.

We point out that in all the above mentioned papers a perfect electrical coupling between the healthy part of the heart and the damaged tissue was assumed.
More general conditions for the heart-torso coupling were proposed in \cite{Boulakia-2007-1}
and investigated through numerical simulations in \cite{Boulakia:2015,Boulakia-2010,Zemzemi:2009}, in order to take into account the possible capacitive and resistive effects of the pericardium.
However, up to our knowledge, there are no rigorous proofs in the literature covering this setting.
We investigate these more general conditions in the context of the bidomain model with diffusive inclusions, where
the appropriate interface behaviour, up to our knowledge, is still not well understood.

The goal of the present paper is to study the well-posedness of such a {\it modified bidomain model}.
We include the structural defects of the heart tissue
in this model by coupling a standard bidomain system in the healthy zone with a diffusion equation posed in the damaged part of the heart,
through non-standard conditions (see equations \eqref{eq:PDEin}--\eqref{eq:Circuit}).
More precisely, for the intra-cellular potential we assume no flux condition on the interface between the two zones
(see \eqref{eq:flux1}), while the extra-cellular
potential is coupled with the electrical potential of the damaged zone through imperfect transmission conditions,
involving the resistive and the capacitive properties of the interface (see \eqref{eq:flux2} and \eqref{eq:Circuit}).

In order to describe the dynamic of the membrane, one can use a physiological ionic model or a phenomenological one (see, for instance, \cite{Collin:Imperiale:2018}). In this paper, the dynamic of the gating variable modeling the ionic transport through the cell membrane is described with the aid of a Hodgkin-Huxley type formalism (see \eqref{eq:gating3}--\eqref{eq:gating2}). Our analysis covers also the modified Mitchell-Schaeffer formalism proposed in \cite{CDP2} (see Remark \ref{r:r2}).

We point out again that
our mathematical model generalizes the modified bidomain model with diffusive inclusions and perfect transmission conditions considered in \cite{CDP2,Davidovic:2016},
the original model being recovered by suitably rearranging the
parameters appearing in equation \eqref{eq:Circuit}. We believe that further numerical simulations have to be carried
out in order to validate the relevance of such transmission conditions also from the point of view of possible biological applications.

The mathematical problem we address here is rather non-standard and, up to our knowledge, the proof of its well-posedness is new in the literature
and generates difficulties due to the degeneracy of the bidomain system and to its special coupling with the diffusion equation.

Our main result is contained in Theorems \ref{p:p2} and \ref{t:t1}, where the leading idea is to reduce the problem to an
abstract parabolic setting (see \cite{Cazenave:Haraux:1998,Pazy:2012}). This requires to introduce several auxiliary differential systems
and a non-standard bilinear form (see Proposition \ref{p:p4}).

The problem proposed here can be seen as a mesoscopic model which will be analyzed in the homogenization limit
in a forthcoming article (see \cite{Amar:Andreucci:Timofte:2020}).
\medskip

The paper is organized as follows: in Section 2, we introduce the mathematical description of our modified bidomain model, together
with its geometrical and functional setting. In Section 3, we state and prove our main result.
\\

\section{The model}\label{s:threeD_problem}

\subsection{Geometrical setting}\label{ss:geometric}
Let $N\geq 3$.
Let $\Om$ be an open connected bounded subset of $\RN$; we assume that $\partial\Om$ is of class $\CC^\infty$,
though this assumption can be weakened. Moreover, for $T>0$, we set $\Om_T=\Om\times(0,T)$.
We assume that $\Om=\Omint\cup\Omout\cup\Memb$, where $\Omint$ and $\Omout$ are
two disjoint open subsets of $\Om$ and
$\Memb=\partial\Omint\cap\Om=\partial\Omout\cap\Om$.
The domain $\Omega$ is occupied by the cardiac tissue, $\Omega^B$ represents the healthy part of the heart tissue, modeled with the aid of a standard bidomain system, $\Omega^D$ represents the diffusive region, accounting for the damaged part of the heart, and $\Gamma$ is the common boundary of these two regions, assumed to be Lipschitz. From a geometrical point of view, we assume that $\Omega^B$ is connected, while $\Omega^D$ might be connected or disconnected.
Indeed, we will consider two different cases: in the first one (to which we will
refer as the {\it connected/disconnected case}, see Fig.\ref{fig:mb} on the left),
we will assume $\Omint\subset\subset\Om$ and $\Omint$ is made by a finite number of connected components.
In this case, $\Memb=\partial\Omint$ and $\partial\Omout\cap\partial\Om\not=\emptyset$.
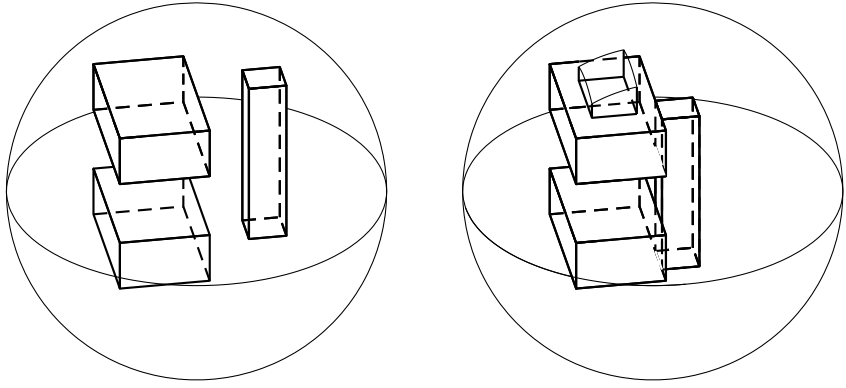
\begin{figure}[htbp]%
\begin{center}%
\input{fig_mb}%
    \caption{\textsl{On the left:} the connected/disconnected case. \textsl{On the right:} the connected/connected case.}
    \label{fig:mb}
  \end{center}
\end{figure}

In the second case (to which we will refer as the {\it connected/connected case},
see Fig.\ref{fig:mb} on the right), we will
assume that both $\Omint$ and $\Omout$ are connected, with $\partial\Omout\cap\partial\Om\not=\emptyset$ and
$\partial\Omint\cap\partial\Om\not=\emptyset$.
Finally, let $\nu$ denote the normal unit vector to $\Memb$ pointing into $\Omout$.

In the following, by $\const$ we shall denote a strictly positive constant,
which may depend on the geometry and on the other parameters of the problem; $\const$ may vary from line to line.
\medskip

\subsection{Functional spaces}\label{ss:spaces}
Let us introduce the following functional spaces:
\begin{equation}
\label{eq:space1}
\begin{aligned}
& H^1_{null}(\Omout):= \{w\in H^1(\Omout): \hbox{$w=0$ on $\partial\Omout\cap\partial\Om$, in the sense of traces}\};
\\
& H^1_{null}(\Omint):= \{w\in H^1(\Omint): \hbox{$w=0$ on $\partial\Omint\cap\partial\Om$, in the sense of traces}\};
\\
& H^{1/2}_0(\Memb,\Om):= \{r\in H^{1/2}(\Memb): \hbox{$r=\widetilde r_{\mid_{\Memb}}$, with $\widetilde r\in H^1_0(\Om)$}\}.
\end{aligned}
\end{equation}
Notice that $H^{1/2}_0(\Memb,\Om)$ is a Hilbert space and, in the connected/disconnected case, $H^{1/2}_0(\Memb,\Om)=H^{1/2}(\Memb)$
and $H^1_{null}(\Omint)=H^1(\Omint)$.

We also set
\begin{equation}\label{eq:space2}
W: = H^1_{null}(\Omout)\times H^{1/2}_0(\Memb,\Om);
\qquad   H:= L^2(\Omout)\times L^2(\Memb),
\end{equation}
where $H$ is endowed with the scalar product
\begin{equation}\label{eq:scalar1}
\big((w,r),(\overline w,s)\big)_{H}
:=\int_{\Omout}w\overline w\di x+\alpha\int_{\Memb}rs\di\sigma\,;
\end{equation}
here, $\alpha>0$ will be the constant appearing later in \eqref{eq:Circuit}
and $W$ is endowed with the scalar product
\begin{equation}\label{eq:scalar2}
\big((w,r),(\overline w,s)\big)_{W}
:=\int_{\Omout}\nabla w\cdot\nabla\overline w\di x+(r,s)_{1/2}\,.
\end{equation}
We denote by $(\cdot,\cdot)_{1/2}$ the standard scalar product on $H^{1/2}(\Memb)$.

Moreover, we define the space
\begin{equation}\label{eq:a1}
\X^1_0(\Om):=\{\W:\Om\to \R\ :\ \W_{\mid_{\Omout}}\in H^1_{null}(\Omout),\
\W_{\mid_{\Omint}}\in H^1_{null}(\Omint)\}\,,
\end{equation}
endowed with the norm
\begin{equation}\label{eq:a3}
\Vert \W\Vert^2_{\X^1_0(\Om)}:= \Vert \nabla\W\Vert^2_{L^2(\Omout)}+\Vert\W\Vert^2_{H^1(\Omint)}.
\end{equation}
We recall that $\partial\Omout\cap\partial\Om$ is always non-empty, while $\partial\Omint$
can intersect or not the boundary of $\Om$, depending on the geometry.
For $\W\in \X^1_0(\Om)$, we have the following Poincar\'{e} inequality
(see \cite[Proposition 2]{Amar:Andreucci:Bisegna:Gianni:2005}):
\begin{equation}\label{eq:poincare1}
\Vert \W\Vert^2_{L^2(\Om)}\leq \const \left(\Vert \nabla\W\Vert^2_{L^2(\Omout)}+\Vert \nabla\W\Vert^2_{L^2(\Omint)}
+\Vert[\W]\Vert^2_{L^2(\Memb)}\right)\,,
\end{equation}
where $[\W]= \W_{\mid_{\Omout}}- \W_{\mid_{\Omint}}$ and
the last term is not necessary in the connected/connected case.
Therefore, an equivalent norm on $\X^1_0(\Om)$ is given by
\begin{equation}\label{eq:a7}
\Vert \W\Vert^2_{\X^1_0(\Om)}\sim \Vert \nabla\W\Vert^2_{L^2(\Om)}+\Vert[\W]\Vert^2_{L^2(\Memb)};
\end{equation}
again, the last term can be dropped in the connected/connected case.

\subsection{Position of the problem}\label{ss:position}

Let $\capuno,\capdue$ be strictly positive constants and
$\dfint,\dfout,\dfdis$ be measurable functions such that
$\const_0\leq \dfint(x),\dfout(x),\dfdis(x)\leq \widetilde\const_0$, a.e. in $\Om$, for
suitable strictly positive constants $\const_0,\widetilde\const_0$.
The assumption that $\dfint,\dfout,\dfdis$ are scalar functions is used only in Section \ref{s:exist}.
Removing this assumption is not trivial. If we want to consider general bounded and symmetric matrices satisfying
\begin{equation}\label{eq:matrix}
\begin{aligned}
&\const_0|\zeta|^2\leq \dfint(x) \zeta\cdot\zeta\leq\widetilde\const_0|\zeta|^2,\qquad\text{for every
$\zeta\in\RN$ and a.e. $x\in \Omout$;}
\\
&\const_0|\zeta|^2\leq \dfout(x) \zeta\cdot\zeta\leq\widetilde\const_0|\zeta|^2,\qquad\text{for every
$\zeta\in\RN$ and a.e. $x\in \Omout$;}
\\
&\const_0|\zeta|^2\leq \dfdis(x) \zeta\cdot\zeta\leq\widetilde\const_0|\zeta|^2,\qquad\text{for every
$\zeta\in\RN$ and a.e. $x\in\Omint$,}
\end{aligned}
\end{equation}
we have to require some other structural hypotheses as in \cite[Lemma 1]{Bourgault:Coudiere:Pierre:2009}
and \cite[Formula (1)]{K2020} (see, also, \cite{BK,Boulakia:2015,Davidovic:2016}).
\medskip

Let us consider a locally Lipschitz continuous function $g:\R^2\to\R$, such that $g(p,1)\geq 0$ and $g(p,0)\leq 0$. The example
we have in mind here is a function of the form
\begin{equation}\label{eq:gating3}
g(p,q)= \aaa(p)( q-1) +\bbb(p)q,
\end{equation}
where $\aaa,\bbb:\R\to\R$ are positive, bounded and Lipschitz functions. Notice that the form of $g$ in \eqref{eq:gating3} is classical
in this framework (see, for instance, \cite{Veneroni:2009}) and that $g$ is Lipschitz continuous with
respect to $p$ and affine with respect to $q$.
Let $\Iion:\R^2\to\R$ be given by
\begin{equation}\label{eq:ion3}
I_{\textrm{ion}}(p,q)= h_1(p) +h_2(p)q,
\end{equation}
where $h_1,h_2$ are Lipschitz continuous functions and $h_2$ is bounded.
Let $\win\in L^\infty(\Omout)$, with $0\leq\win(x)\leq 1$ a.e. in $\Omout$, and
$p\in L^2(\Omout_T)$.
Consider
the gating equation
\begin{alignat}2
\label{eq:gating1}
&\partial_t \widetilde w_p+g(p,\widetilde w_p)=0,\qquad &\text{in $\Omout_T$;}
\\
  \label{eq:gating2}
& \widetilde w_p(x,0)=\win(x),\qquad &\text{in $\Omout$.}
\end{alignat}
Notice that, by classical results, the previous problem admits a unique solution $\widetilde w_p\in H^1(0,T;L^\infty(\Omout))$
and, from our assumptions, $0\leq\widetilde w_p(x,t)\leq 1$ a.e. in $\Omout_T$, since
$0\leq\win(x)\leq 1$ a.e. in $\Omout$. This is a standard result for ODEs, taking into account that the spatial variable plays
here only the role of a parameter (for similar results, see, for instance, \cite{Collin:Imperiale:2018,CDP2,JPr}).

Moreover, from the previous assumptions, we can prove that
there exists a strictly positive constant $\const$ such that
\begin{equation}
\label{eq:ion1}
\big\Vert \Iion(p_1,\widetilde w_{p_1})-\Iion(p_2,\widetilde w_{p_2})\big\Vert_{L^2(\Omout_T)}\leq \const_I\Vert p_1-p_2\Vert_{L^2(\Omout_T)}\,,
\end{equation}
due to the Lipschitz dependence of $\widetilde w_p$ on $p$ and to the bound $0\leq\widetilde w_p(x,t)\leq 1$ a.e. in $\Omout_T$.

We give here a complete formulation of the problem we shall address in this paper.
The operators $\Div$ and $\nabla$ act
only with respect to the space variable $x$.

Let $f_1,f_2\in H^1(\Omout_T)$, $\overline v_{0}\in L^2(\Omout)$, $s_0\in L^2(\Memb)$ and
consider the problem
for $\ubuno,\ubdue\in L^2(0,T;H^1_{null}(\Omout)), \ud\in L^2(0,T;H^1_{null}(\Omint))$ and
$\widetilde w\in H^1(0,T;L^\infty(\Omout))$
given by
\begin{alignat}2
  \label{eq:PDEin}
&\frac{\partial}{\partial t}(\ubuno-\ubdue)-\Div(\dfint \nabla \ubuno)+\Iion(\ubuno-\ubdue,\widetilde w)\!=\!f_1,\ &\text{in $\Omout_T$;}
  \\
  \label{eq:PDEout}
&\frac{\partial}{\partial t}(\ubuno-\ubdue)+\Div(\dfout \nabla \ubdue)+\Iion(\ubuno-\ubdue,\widetilde w)\!=\!f_2,\ &\text{in $\Omout_T$;}
  \\
  \label{eq:PDEdis}
&-\Div(\dfdis \nabla \ud)=0,&\text{in $\Omint_T$;}
\\
\label{eq:flux1}
&  \dfint\nabla\ubuno\cdot\nu=0,&\text{on $\Memb_T$;}
\\
\label{eq:flux2}
&  \dfout\nabla\ubdue\cdot\nu =\dfdis\nabla\ud\cdot\nu,&\text{on $\Memb_T$;}
\\
  \label{eq:Circuit}
&  \capuno\frac{\partial}{\partial t}(\ubdue-\ud)+\capdue  (\ubdue-\ud)=\dfout\nabla \ubdue\cdot\nu,
&\text{on $\Memb_T$;}
  \\
  \label{eq:BoundData}
&  \ubuno(x,t),\ubdue(x,t),\ud(x,t)=0,&\text{on $\partial\Om\times(0,T)$;}
  \\
  \label{eq:InitData1}
& \ubuno(x,0)-\ubdue(x,0)=\overline v_{0}(x),&\text{in $\Omout$;}
\\
  \label{eq:InitData3}
& \ubdue(x,0)-\ud(x,0)=s_{0}(x),&\text{on $\Memb$,}
\end{alignat}
where $\widetilde w$ is the solution of the gating equation \eqref{eq:gating1}, \eqref{eq:gating2},
with $p=\ubuno-\ubdue$.

\begin{remark}[Biological interpretation]\label{r:r3}
In the previous system of equations, the coefficients $\dfint,\dfout$ and $\dfdis$ are
the conductivities of the two healthy phases and of the damaged one, respectively,
while $\capuno$ and $\capdue$ are given parameters related to the capacitive and the resistive behaviour of the interface $\Memb$.
The functions $f_1$ and $f_2$, appearing in \eqref{eq:PDEin} and \eqref{eq:PDEout}, respectively,
represent the internal and the external current stimulus.
The solutions $\ubuno$ and $\ubdue$ are the intra and the extra-cellular potentials of the healthy zone, while
$\ud$ is the electrical potential of the damaged zone. The function $\ubdue-\ubuno$ is the so-called {\it transmembrane potential}.
Finally, the variable $\widetilde w$, called the {\it gating variable}, describes the ionic transport through the cell membrane.
The terms $g$ and $I_{ion}$ are nonlinear functions, modeling the membrane ionic currents.
For simplicity, we consider only one gating variable, but our results hold true also for the case in which the gating variable is vector valued.
\end{remark}

\begin{remark}\label{r:r4}
In equation \eqref{eq:flux1}, we are assuming that the intra-cellular potential satisfies a no flux condition
on the interface $\Memb_T$ between the two zones $\Omout$ and $\Omint$,
while equations \eqref{eq:flux2} and \eqref{eq:Circuit} state the imperfect
transmission conditions coupling the extra-cellular potential with the electrical potential of the damaged zone $\Omint$. Such conditions
involve the capacitive and the resistive properties of the interface, represented by the constants $\alpha$ and $\beta$, respectively.
\end{remark}

\begin{remark}\label{r:r2}
Different examples of functions $\Iion$ and $g$ are considered in the literature. We consider here
a Hodgkin-Huxley type model (see \eqref{eq:gating3}-\eqref{eq:ion3}), as in \cite{Amar:Andreucci:Timofte:2020,Collin:Imperiale:2018,Veneroni:2009}.
However, we point out that the results obtained in this paper are also valid for a regularized version of the Mitchell-Schaeffer model proposed in
\cite{CDP2} (see, also, \cite{CDP1,Davidovic:2016,KLPW}).
\end{remark}

By standard approximation procedure,
multiplying \eqref{eq:PDEin} by $\ubuno$, \eqref{eq:PDEout} by $\ubdue$, \eqref{eq:PDEdis} by $\ud$,
subtracting \eqref{eq:PDEout} from \eqref{eq:PDEin}, adding \eqref{eq:PDEdis},
integrating by parts, using \eqref{eq:flux1}--\eqref{eq:InitData3}, \eqref{eq:ion3}, \eqref{eq:ion1} and
moving the integral containing $\Iion$ to the right-hand side, we get
\begin{multline}\label{eq:energy1}
\frac{1}{2}\int_{\Omout} (\ubuno-\ubdue)^2(x,T)\di x +\int_{\Omout_T}\dfint\vert\nabla\ubuno\vert^2\di x\di t
+\int_{\Omout_T}\dfout\vert\nabla \ubdue\vert^2\di x\di t
\\
+\int_{\Omint_T}\dfdis\vert\nabla \ud\vert^2\di x\di t
+\frac{\capuno}{2} \int_{\Memb}(\ubdue-\ud)^2(x,T)\di \sigma(x)
+\capdue\int_{\Memb_T} (\ubdue-\ud)^2(x,t)\di \sigma(x)\di t
\\
=
\int_{\Omout_T}(f_1\ubuno-f_2\ubdue)\di x\di t+ \frac{1}{2}\int_{\Omout} \overline v_0^2(x)\di x
+\frac{\capuno}{2} \int_{\Memb}s_0^2(x)\di \sigma(x)
\\
-\int_{\Omout_T}\Iion(\ubuno-\ubdue,\widetilde w)(\ubuno-\ubdue)\di x\di t
\\
\leq
\int_{\Omout_T}(f_1\ubuno-f_2\ubdue)\di x\di t+ \const(\Vert \overline v_0\Vert^2_{L^2(\Omout)}
+\Vert s_0\Vert^2_{L^2(\Memb)})
\\
-\int_{\Omout_T}\Big(\Iion(\ubuno-\ubdue,\widetilde w)-\Iion(0,w_0)\Big)(\ubuno-\ubdue)\di x\di t
\\
-\int_{\Omout_T}\Iion(0,w_0)(\ubuno-\ubdue)\di x\di t
\\
\leq \const(\Vert f_1\Vert^2_{L^2(\Omout_T)}+\Vert f_2\Vert^2_{L^2(\Omout_T)}+\Vert \overline v_0\Vert^2_{L^2(\Omout)}
+\Vert s_0\Vert^2_{L^2(\Memb)}+1)
\\
+\frac{\delta}{2} (\Vert \nabla \ubuno\Vert^2_{L^2(\Omout_T)}
+\Vert \nabla \ubdue\Vert^2_{L^2(\Omout_T)})
+\const\Vert \ubuno- \ubdue\Vert^2_{L^2(\Omout_T)})\,,
\end{multline}
where $\const$ and $\delta$ are positive constants,
$\delta$ can be chosen smaller than $\min(\dfint,\dfout)$,
and we have also applied Poincar\'{e} inequality to $\ubuno$ and $\ubdue$.
By absorbing into the left-hand side the first two terms in the last line of \eqref{eq:energy1}
and using Gronwall inequality, from the previous estimate, we obtain
\begin{multline}\label{eq:energy5}
\sup_{t\in(0,T)}\int_{\Omout} (\ubuno-\ubdue)^2(x,t)\di x +\int_{\Omout_T}\vert\nabla\ubuno\vert^2\di x\di t
+\int_{\Omout_T}\vert\nabla \ubdue\vert^2\di x\di t
+\int_{\Omint_T}\vert\nabla \ud\vert^2\di x\di t
\\
+\sup_{t\in(0,T)}\int_{\Memb}(\ubdue-\ud)^2(x,t)\di \sigma(x)
+\int_{\Memb_T} (\ubdue-\ud)^2(x,t)\di \sigma(x)\di t
\\
\leq \const(\Vert f_1\Vert^2_{L^2(\Omout_T)}+\Vert f_2\Vert^2_{L^2(\Omout_T)}+\Vert \overline v_0\Vert^2_{L^2(\Omout)}
+\Vert s_0\Vert^2_{L^2(\Memb)}+1)\,.
\end{multline}

\begin{prop}
Under the assumptions stated above, problem \eqref{eq:PDEin}--\eqref{eq:InitData3} and \eqref{eq:gating1}-\eqref{eq:gating2}
admits at most one solution.
\end{prop}

\begin{proof}
Assume that $(\ubuno,\ubdue, \ud,\widetilde w)$ and $(\overlineubuno,\overlineubdue,
\overlineud,\overline{\widetilde w})$ are two different solutions of problem
\eqref{eq:PDEin}-\eqref{eq:InitData1}, with $\overline{\widetilde w}$ being the solution of \eqref{eq:gating1}-\eqref{eq:gating2},
corresponding to $p=\overlineubuno-\overlineubdue$. Setting $\upsilonuno:=\ubuno-\overlineubuno,
\upsilondue:=\ubdue-\overlineubdue$ and $ \upsilond:=\ud-\overlineud$,
we obtain that $\upsilonuno,\upsilondue, \upsilond,\widetilde w$ and $\overline{\widetilde w}$ solve the system
\begin{equation*}
\begin{aligned}
&\frac{\partial}{\partial t}(\upsilonuno-\upsilondue)-\Div(\dfint \nabla \upsilonuno)+\Iion(\ubuno-\ubdue,\widetilde w)
-\Iion(\overlineubuno-\overlineubdue,\overline{\widetilde w})\!=\!0,\!\!\!\!\!\!\!\!\!\!\!\!\!\!\!\!\!\!\!\!\!\!\!\!\!&\!\!\!\!\!\!\!\!\!
\!\!\!\!\!\!\!\!\!\!\!\!\!\!\!\!\!\!\!\!\!\!\!\!\!\!\!\!\!\!\!\!\text{in $\Omout_T$;\ }
  \\
&\frac{\partial}{\partial t}(\upsilonuno-\upsilondue)+\Div(\dfout \nabla \upsilondue)+
\Iion(\ubuno-\ubdue,\widetilde w)-\Iion(\overlineubuno-\overlineubdue,\overline{\widetilde w})\!=\!0,
\!\!\!\!\!\!\!\!\!\!\!\!\!\!\!\!\!\!\!\!&\!\!\!\!
\!\!\!\!\!\!\!\!\!\!\!\!\!\!\!\!\!\!\!\text{in $\Omint_T$;\ }
  \\
&-\Div(\dfdis \nabla \upsilond)=0,&\text{in $\Omint_T$;\ }
\\
& \dfint\nabla\upsilonuno\cdot\nu=0,&\text{on $\Memb_T$;\ }
\\
& \dfout\nabla\upsilondue\cdot\nu =\dfdis\nabla\upsilond\cdot\nu,&\text{on $\Memb_T$;\ }
\\
& \capuno\frac{\partial}{\partial t}(\upsilondue-\upsilond)+\capdue  (\upsilondue-\upsilond)=\dfout\nabla \upsilondue\cdot\nu,
&\text{on $\Memb_T$;\ }
  \\
&  \upsilonuno(x,t),\upsilondue(x,t),\upsilond(x,t)=0,&\text{on $\partial\Om\times(0,T)$;\ }
  \\
& \upsilonuno(x,0)-\upsilondue(x,0)=0,&\text{in $\Omout$;\ }
\\
& \upsilondue(x,0)-\upsilond(x,0)=0,&\text{on $\Memb$.\ }
\end{aligned}
\end{equation*}
Reasoning in a similar way as done for \eqref{eq:energy1}, i.e.
by multiplying the first equation by $\upsilonuno$,
the second one by $\upsilondue$, the third one by $\upsilond$,
subtracting the second equation from the first one, adding
the third one, integrating by parts, using the remaining equation of the previous system
moving the integral containing $\Iion$ to the right-hand side and using H\"{o}lder inequality, we get
\begin{multline*}
\sup_{t\in(0,T)}\int_{\Omout} (\upsilonuno-\upsilondue)^2(x,t)\di x +\int_{\Omout_T}\vert\nabla\upsilonuno\vert^2\di x\di t
+\int_{\Omout_T}\vert\nabla \upsilondue\vert^2\di x\di t
+\int_{\Omint_T}\vert\nabla \upsilond\vert^2\di x\di t
\\
+\sup_{t\in(0,T)}\int_{\Memb}(\upsilondue-\upsilond)^2(x,t)\di \sigma(x)
+\int_{\Memb_T} (\upsilondue-\upsilond)^2(x,t)\di \sigma(x)\di t
\\
\leq
\Vert\Iion(\ubuno-\ubdue,\widetilde w)-\Iion(\overlineubuno-\overlineubdue,\overline{\widetilde w)}\Vert_{L^2(\Omout_T)}\cdot
\Vert\upsilonuno-\upsilondue\Vert_{L^2(\Omout_T)}
\leq \const\Vert \upsilonuno- \upsilondue\Vert^2_{L^2(\Omout_T)}\,,
\end{multline*}
where, in the last inequality, we used \eqref{eq:ion1}. We can conclude by using Gronwall inequality.
\end{proof}

Notice that, by setting $V=\ubuno-\ubdue$, $U=\ubdue$ a.e. in $\Omout$, $U=\ud$ a.e. in $\Omint$,
and denoting by $[\cdot]$ the jump across $\Permemb$ of the quantity in the square brackets, i.e.
$[U]=\ubdue-\ud$ and $[\dfboth\nabla U\cdot\nu]=(\dfout\nabla \ubdue-\dfdis\nabla \ud)\cdot\nu$,
the system \eqref{eq:PDEin}--\eqref{eq:InitData3} can be written in the more convenient form
\begin{alignat}2
  \label{eq:PDEinc}
&\frac{\partial V}{\partial t}-\Div(\dfint \nabla V)+\Iion(V,\widetilde w)\!=\!f_1+\Div(\dfint\nabla U),\ &\text{in $\Omout_T$;}
  \\
  \label{eq:PDEoutc}
&-\Div((\dfint+\dfout) \nabla U)\!=\!f_1-f_2+\Div(\dfint\nabla V),\ &\text{in $\Omout_T$;}
  \\
  \label{eq:PDEdisc}
&-\Div(\dfdis \nabla U)=0,&\text{in $\Omint_T$;}
\\
\label{eq:flux1c}
&  \dfint\nabla ( V+U)\cdot\nu =0,&\text{on $\Memb_T$;}
\\
\label{eq:flux2c}
&  [\dfboth\nabla U\cdot\nu] =0,&\text{on $\Memb_T$;}
\\
  \label{eq:Circuitc}
&  \capuno\frac{\partial}{\partial t}[U]+\capdue  [U]=\dfout\nabla U\cdot\nu,
&\text{on $\Memb_T$;}
  \\
  \label{eq:BoundDatac}
&  V, U=0,&\text{on $\partial\Om\times(0,T)$;}
  \\
  \label{eq:InitData1c}
& V(x,0)=\overline v_{0}(x),&\text{in $\Omout$;}
\\
  \label{eq:InitData3c}
& [U](x,0)=s_{0}(x),&\text{on $\Memb$,}
\end{alignat}
complemented with the gating problem \eqref{eq:gating1}-\eqref{eq:gating2}, where again
$\ubuno-\ubdue$ is replaced by $V$. Clearly, $V\in L^2(0,T;H^1_{null}(\Omout))$ and $U\in L^2(0,T;\X^1_0(\Om))\underline{}$.
We recall that, by \eqref{eq:ion1}, the composed function $\Iion(V, \widetilde w)$
is also a Lipschitz function with respect to $V$.

The weak formulation of the previous problem is given by
\begin{multline}\label{eq:weak5}
-\int_{\Omout_T} V\partial_t\testb\di x\di t+\int_{\Omout_T}\dfint\nabla V\cdot\nabla \testb\di x\di t
+\int_{\Omout_T}\dfint\nabla U\cdot\nabla\testb
+\int_{\Omout_T} \Iion(V,\widetilde w)\testb\di x\di t
\\
+\int_{\Omout_T}(\dfint+\dfout)\nabla U\cdot\nabla \testduno \di x\di t
+\int_{\Omout_T}\dfint\nabla V\cdot\nabla \testduno\di x\di t
\\
+\int_{\Omint_T}\dfdis\nabla U\cdot \nabla\testddue\di x\di t
-\capuno\int_{\Memb_T}[U]\partial_t[\testd]\di\sigma\di t
+\capdue\int_{\Memb_T}[U][\testd]\di\sigma\di t
\\
=\!\!\int_{\Omout_T}f_1\testb\di x\di t+\int_{\Omout_T}(f_1-f_2)\testduno\di x\di t\,,
+\int_{\Omout} \overline v_0\testb(0)\di x
+\capuno\int_{\Memb}s_0[\testd](0)\di\sigma,
\end{multline}
for every $\testb\in L^2(0,T;H^1_{null}(\Omout))\cap H^1(0,T;L^2(\Omout))$,
$\testd\in L^2(0,T;\X^1_0(\Om))$ (which means
$\testduno\in L^2(0,T;H^1_{null}(\Omout))$, $\testddue\in L^2(0,T;H^1_{null}(\Omint))$), and $[\testd]\in H^1(0,T;L^2(\Memb))$,
with $\testb(T)=0$ and $[\testd](T)=0$.
Here, $[\testd]=\testduno-\testddue$ on $\Memb$.
We remark that the notation $\testduno$ for a test function acting in $\Omout$ is slightly counterintuitive, but
it allows us to write concisely the jump of such a function across the interface.
Clearly, by \eqref{eq:energy5}, we get the following energy inequality:
\begin{multline}\label{eq:energy3}
\sup_{t\in (0,T)}\int_{\Omout}V^2(x,t)\di x +\int_{\Omout_T}|\nabla V+\nabla U|^2\di x\di t
+\int_{\Omout_T}|\nabla U|^2\di x\di t
\\
+\int_{\Omint_T}|\nabla U|^2\di x\di t+\sup_{t\in (0,T)}\int_{\Memb}[U]^2(x,t)\di\sigma
+\int_{\Memb_T}[U]^2\di \sigma\di t
\\
\leq\const\left(\Vert f_1\Vert^2_{L^2(\Omout_T)}+\Vert f_2\Vert^2_{L^2(\Omout_T)}+\Vert \overline v_0\Vert^2_{L^2(\Omout)}
+\Vert s_0\Vert^2_{L^2(\Memb)}+1\right)\,,
\end{multline}
where $\const$ depends on $\const_0,\const_I,\capuno,\capdue$, and the geometry.
Notice that, by \eqref{eq:energy3}, it follows also that
\begin{equation}\label{eq:energy4}
\int_{\Omout_T}|\nabla V\vert^2\di x\di t\leq \const\left(\Vert f_1\Vert^2_{L^2(\Omout_T)}+\Vert f_2\Vert^2_{L^2(\Omout_T)}+
\Vert \overline v_0\Vert^2_{L^2(\Omout)}+\Vert s_0\Vert^2_{L^2(\Memb)}+1\right).
\end{equation}

\section{Well-posedness}\label{s:exist}

We will consider our problem in an abstract setting and to this
purpose we need first ``to move'' the source $f_1-f_2$ from \eqref{eq:PDEoutc} to \eqref{eq:PDEinc}, \eqref{eq:Circuitc}, and
\eqref{eq:InitData3c}.
Then, we introduce a bilinear form $a$ on
$ W\times W$ such that the problem is reduced to the abstract scheme
\eqref{eq:problem}. This form is constructed with auxiliary functions which are
obtained in Proposition \ref{p:p1} and in Remark \ref{r:r1}. In Proposition \ref{p:p4}, we
prove the necessary properties of the form $a$. In Theorem \ref{p:p2} and
Proposition \ref{p:p3}, we prove existence of solutions to the problem with
$I_{ion}\equiv 0$. Finally, the full result is obtained in Theorem \ref{t:t1},
where the complete problem will be treated as a nonlinear perturbation
of this case (see, for instance, \cite{Cazenave:Haraux:1998,CDP2,JPr,Pazy:2012}).

We start by considering, for a.e. $t\in (0,T)$,
the following auxiliary problem:
\begin{alignat}2
\label{eq:pde1}
 -\Div\big((\dfint+\dfout)\nabla \widetilde u\big) &= f_1-f_2, &\quad \hbox{in $\Omout$;}
\\
\label{eq:pde2}
\dfint\nabla\widetilde u\cdot\nu &=0, &\quad \hbox{on $\Memb$;}
\\
\label{eq:pde3}
 \widetilde u &= 0, &\quad \hbox{on $\partial\Omout\cap\partial\Om$.}
\end{alignat}
Clearly,
problem \eqref{eq:pde1}--\eqref{eq:pde3} is classical and admits a unique solution
$\widetilde u\in H^1(0,T;H^1_{null}(\Omout))$.
Moreover, we extend $\widetilde u$ inside $\Omint$ by zero, so that it has a nonzero jump on $\Memb$;
i.e., $\widetilde u\in \X^1_0(\Om)$.
On $\Memb$, let us define
\begin{equation}\label{eq:a8}
q:=-\capuno\frac{\partial [\widetilde u]}{\partial t}-\capdue[\widetilde u]
\end{equation}
and consider the problem for $(v,u)\in L^2(0,T;H^1_{null}(\Omout))\times L^2(0,T;\X^1_0(\Om))$ given by
\begin{alignat}2
  \label{eq:PDEins}
&\frac{\partial v}{\partial t}-\Div(\dfint \nabla v)\!=\!f_1+\Div(\dfint\nabla u)
+\Div(\dfint\nabla \widetilde u),\ &\text{in $\Omout_T$;}
  \\
  \label{eq:PDEouts}
&-\Div((\dfint+\dfout) \nabla u)\!=\!\Div(\dfint\nabla v),\ &\text{in $\Omout_T$;}
  \\
  \label{eq:PDEdiss}
&-\Div(\dfdis \nabla u)=0,&\text{in $\Omint_T$;}
\\
\label{eq:flux1s}
&  \dfint\nabla ( v+u)\cdot\nu =0,&\text{on $\Memb_T$;}
\\
\label{eq:flux2s}
&  \dfout\nabla u\cdot\nu =\dfdis\nabla u\cdot\nu,&\text{on $\Memb_T$;}
\\
  \label{eq:Circuits}
&  \capuno\frac{\partial}{\partial t}[u]+\capdue  [u]=\dfout\nabla u\cdot\nu+q,
&\text{on $\Memb_T$;}
  \\
  \label{eq:BoundDatas}
&  v, u=0,&\text{on $\partial\Om\times(0,T)$;}
  \\
  \label{eq:InitData1s}
& v(x,0)=\overline v_{0}(x),&\text{on $\Omout$;}
\\
  \label{eq:InitData3s}
& [u](x,0)=s_{0}(x)-[\widetilde u](x,0),&\text{on $\Memb$.}
\end{alignat}
The weak formulation of the previous problem is given by
\begin{multline}\label{eq:weak1}
-\int_{\Omout_T} v\partial_t\testb\di x\di t+\int_{\Omout_T}\dfint\nabla v\cdot\nabla \testb\di x\di t
+\int_{\Omout_T}\dfint\nabla u\cdot\nabla\testb\di x\di t
\\
+\int_{\Omout_T}(\dfint+\dfout)\nabla u\cdot\nabla \testduno \di x\di t
+\int_{\Omout_T}\dfint\nabla v\cdot\nabla \testduno\di x\di t
\\
+\int_{\Omint_T}\dfdis\nabla u\cdot \nabla\testddue\di x\di t
-\capuno\int_{\Memb_T}[u]\partial_t[\testd]\di\sigma\di t
+\capdue\int_{\Memb_T}[u][\testd]\di\sigma\di t
\\
=\int_{\Omout_T}f_1\testb\di x\di t
-\int_{\Omout_T}\dfint\nabla\widetilde u\cdot\nabla\testb\di x\di t
+\int_{\Memb_T}q[\testd]\di\sigma\di t\,,
\end{multline}
for every $\testb\in L^2(0,T;H^1_{null}(\Omout))\cap H^1_0(0,T;L^2(\Omout))$,
$\testd\in L^2(0,T;\X^1_0(\Om))$, where, as before, $[\testd]=\testduno-\testddue$ and $[\testd]\in H^1_0(0,T;L^2(\Memb))$.
The weak formulation \eqref{eq:weak1} shall be complemented with the initial conditions.
Indeed, as it will be proved in Theorem \ref{p:p2}, we have $v\in \CC^0([0,T];L^2(\Omout))$ and $[u]\in\CC^0([0,T];L^2(\Memb))$.
\medskip

The next step is to define a suitable bilinear form on $W$, which is continuous and coercive. To this purpose, we need the following result.

\begin{prop}\label{p:p1}
Let $(w,r)\in W$ be assigned.
Then, there exists a unique solution $\W\in \X^1_0(\Om)$ of the problem
\begin{alignat}2
  \label{eq:PDEuno}
&-\Div((\dfint+\dfout) \nabla\W )\!=\!\Div(\dfint\nabla w),\qquad &\text{in $\Omout$;}
  \\
  \label{eq:PDEdue}
&-\Div(\dfdis \nabla\W)=0,&\text{in $\Omint$;}
\\
\label{eq:fluxuno}
&  (\dfint+\dfout)\nabla \W\cdot\nu =\dfdis\nabla\W\cdot \nu-\dfint\nabla w\cdot\nu,\qquad
&\text{on $\Memb$;}
\\
  \label{eq:Circuituno}
&  [\W]=r, &\text{on $\Memb$;}
  \\
  \label{eq:Bound}
&  \W=0,&\text{on $\partial\Om$.}
\end{alignat}
Moreover, there exists a constant $\const>0$, depending on $\dfint,\dfout,\dfdis$, and the geometry, such that
\begin{equation}\label{eq:a4}
\Vert \W\Vert_{\X^1_0(\Om)}\leq \const(\Vert w\Vert_{H^1(\Omout)}+\Vert r\Vert_{H^{1/2}(\Memb)})\,.
\end{equation}
\end{prop}

\begin{proof}
Uniqueness for problem \eqref{eq:PDEuno}--\eqref{eq:Bound} is a straightforward consequence of its linearity.
In order to prove that a solution does exist, we first
consider the following auxiliary problem:
\begin{alignat}2
  \label{eq:PDEtre}
&-\Div((\dfint+\dfout) \nabla\overline\W_1 )\!=\!\Div(\dfint\nabla w),\qquad &\text{in $\Omout$;}
  \\
\label{eq:datauno}
&  \overline \W_1=r,&\text{on $\Memb$;}
\\
  \label{eq:PDEtrebis}
&  \overline\W_1=0,&\text{in $\Omint$;}
\\
  \label{eq:datadue}
& \overline\W_1=0, &\text{on $\partial\Om$.}
\end{alignat}
Clearly, the previous problem admits a unique solution $\overline\W_1\in \X^1_0(\Om)$. Moreover,
there exists a constant $\const>0$, depending on $\dfint,\dfout$ and the geometry, such that
\begin{equation}\label{eq:a5}
\Vert \overline\W_1\Vert_{\X^1_0(\Om)}\leq \const(\Vert w\Vert_{H^1(\Omout)}+\Vert r\Vert_{H^{1/2}(\Memb)})\,.
\end{equation}
Indeed, let us denote by $\widetilde r\in H^1_{null}(\Omout)$
an extension of $r$ from $\Memb$ to the whole $\Omout$,
such that $\Vert\widetilde r\Vert_{H^1(\Omout)}\leq \const \Vert r\Vert_{H^{1/2}(\Memb)}$,
and set $\overline \W^r_1:=\overline\W_1-\widetilde r$.
Clearly, $\overline \W_1^r$ satisfies the problem
\begin{alignat}2
  \label{eq:PDEquattro}
&-\Div((\dfint+\dfout) \nabla\overline\W_1^r )\!=\!\Div(\dfint\nabla w)+\Div((\dfint+\dfout) \nabla\widetilde r ),\ &\text{in $\Omout$;}
  \\
  \nonumber &
\\
  \label{eq:dataquattro}
& \overline\W_1^r=0, &\text{on $\partial\Omout$.}
\end{alignat}
Therefore, by the standard energy inequality, we get
\begin{equation}\label{eq:energy2}
\int_{\Omout}\vert \nabla\overline\W_1^r\vert^2\di x\leq\const
\left(\Vert\nabla w\Vert^2_{L^2(\Omout)}+\Vert r\Vert^2_{H^{1/2}(\Memb)}\right),
\end{equation}
which implies \eqref{eq:a5}.

Now, let us consider the second auxiliary problem for $\overline\W_2\in H^1_0(\Om)$ given by
\begin{alignat}2
  \label{eq:PDEcinque}
&-\Div((\dfint+\dfout) \nabla\overline\W_2 )\!=0,\ &\text{in $\Omout$;}
\\
  \label{eq:PDEsei}
&-\Div(\dfdis \nabla\overline\W_2 )\!=0,\ &\text{in $\Omint$;}
\\
\label{eq:jumpcinque}
& [\overline \W_2]=0,\ &\text{on $\Memb$;}
  \\
\label{eq:circuitcinque}
& \big((\dfint+\dfout)\nabla\overline \W_2-\dfdis\nabla\overline\W_2\big)\cdot\nu=
-\big((\dfint+\dfout)\nabla\overline\W_1+\dfint\nabla w\big)\cdot\nu,\
&\text{on $\Memb$;}
  \\
\label{eq:datacinque}
&  \overline \W_2=0, &\text{on $\partial\Om$.}
\end{alignat}
Existence and uniqueness for the previous problem is guaranteed by \cite[Lemma 5]{Amar:Andreucci:Bisegna:Gianni:2005};
moreover, the weak formulation of \eqref{eq:PDEcinque}--\eqref{eq:datacinque} is given by
\begin{multline}\label{eq:weak2}
0=\int_{\Omout}(\dfint+\dfout)\nabla\overline\W_2\cdot\nabla\varphi\di x
+\int_{\Memb}(\dfint+\dfout)\nabla\overline\W_2\cdot\nu\,\varphi\di \sigma
\\
+\int_{\Omint}\dfdis\nabla\overline\W_2\cdot\nabla\varphi\di x
-\int_{\Memb}\dfdis\nabla\overline\W_2\cdot\nu\,\varphi\di \sigma
\\
=\int_{\Omout}(\dfint+\dfout)\nabla\overline\W_2\cdot\nabla\varphi\di x
+\int_{\Omint}\dfdis\nabla\overline\W_2\cdot\nabla\varphi\di x
\\
-\int_{\Memb}(\dfint+\dfout)\nabla\overline\W_1\cdot\nu\,\varphi\di\sigma
-\int_{\Memb}\dfint\nabla w\cdot\nu\,\varphi\di\sigma\,,
\end{multline}
for every $\varphi\in H^1_0(\Om)$.
From \eqref{eq:PDEtre}, we obtain
\begin{multline*}
-\int_{\Memb}(\dfint+\dfout)\nabla\overline\W_1\cdot\nu\,\varphi\di\sigma
-\int_{\Memb}\dfint\nabla w\cdot\nu\,\varphi\di\sigma
\\
=\int_{\Omout}(\dfint+\dfout)\nabla\overline\W_1\cdot\nabla\varphi\di x
+\int_{\Omout}\dfint\nabla w\cdot\nabla\varphi\di x\,.
\end{multline*}
which, replaced in \eqref{eq:weak2}, provides
\begin{multline}\label{eq:weak3}
0=\int_{\Omout}(\dfint+\dfout)\nabla\overline\W_2\cdot\nabla\varphi\di x
+\int_{\Omint}\dfdis\nabla\overline\W_2\cdot\nabla\varphi\di x
\\
+\int_{\Omout}(\dfint+\dfout)\nabla\overline\W_1\cdot\nabla\varphi\di x
+\int_{\Omout}\dfint\nabla w\cdot\nabla\varphi\di x\,.
\end{multline}
By taking $\varphi =\overline\W_2$ in \eqref{eq:weak3}, we get
\begin{multline*}
\int_{\Omout}(\dfint+\dfout)\vert\nabla\overline\W_2\vert^2\di x+\int_{\Omint}\dfdis\vert\nabla \overline\W_2\vert^2\di x
\\
=
-\int_{\Omout}(\dfint+\dfout)\nabla\overline\W_1\cdot\nabla\overline\W_2\di x
-\int_{\Omout}\dfint\nabla w\cdot\nabla\overline\W_2\di x\,,
\end{multline*}
which, taking into account \eqref{eq:a5}, implies
\begin{equation}\label{eq:a2}
\Vert\nabla\overline\W_2\Vert_{L^2(\Om)}\leq\const(\Vert\nabla\overline\W_1\Vert_{L^2(\Omout)}+\Vert\nabla w\Vert_{L^2(\Omout)})
\leq \const(\Vert w\Vert_{H^1(\Omout)}+\Vert r\Vert_{H^{1/2}(\Memb)})\,,
\end{equation}
with $\const$ depending only on $\dfint,\dfout,\dfdis$ and the geometry.

Finally, setting $\W=\overline\W_1+\overline \W_2$, it is easy to see that $\W\in \X^1_0(\Om)$ and satisfies \eqref{eq:PDEuno}--\eqref{eq:Bound}
and also \eqref{eq:a4}.
\end{proof}

\begin{remark}\label{r:r1}
We point out that, taking $\widetilde r\in H^1_{null}(\Omout)$ as in the proof of Proposition \ref{p:p1} and defining
\begin{equation}\label{eq:a10}
\W^r=\left\{
\begin{aligned}
& \W-\widetilde r\,,\qquad &\hbox{in $\Omout$;}
\\
& \W\,,\qquad &\hbox{in $\Omint$,}
\end{aligned}
\right.
\end{equation}
we get that $\W^r\in H^1_0(\Om)$ and satisfies
\begin{alignat}2
  \label{eq:PDEunod}
&-\Div((\dfint+\dfout) \nabla\W^r )\!=\!\Div(\dfint\nabla w)+\Div((\dfint+\dfout)\nabla\widetilde r),\qquad &\text{in $\Omout$;}
  \\
  \label{eq:PDEdued}
&-\Div(\dfdis \nabla\W^r)=0,&\text{in $\Omint$;}
\\
\label{eq:fluxunod}
&  (\dfint+\dfout)\nabla \W^r\cdot\nu =\dfdis\nabla\W^r\!\cdot\! \nu-\dfint\nabla w\!\cdot\!\nu
-(\dfint+\dfout)\nabla\widetilde r\cdot\nu,\qquad
&\text{on $\Memb$;}
\\
  \label{eq:Circuitunod}
&  [\W^r]=0, &\text{on $\Memb$;}
  \\
  \label{eq:Boundd}
&  \W^r=0,&\text{on $\partial\Om$.}
\end{alignat}
The weak formulation of problem \eqref{eq:PDEunod}--\eqref{eq:Circuitunod} is given by
\begin{multline}\label{eq:weak6}
\int_{\Omout}(\dfint+\dfout)\nabla\W^r\cdot\nabla\phi\di x+\int_{\Omint} \dfdis\nabla\W^r\cdot\nabla\phi
\\
= -\int_{\Omout}\dfint\nabla w\cdot\nabla\phi\di x -\int_{\Omout}(\dfint+\dfout)\nabla\widetilde r\cdot\nabla\phi\di x\,,
\end{multline}
for every $\phi\in H^1_0(\Om)$, which implies
\begin{equation}\label{eq:weak7}
\int_{\Omout}(\dfint+\dfout)\nabla\W\cdot\nabla\phi\di x+\int_{\Omint} \dfdis\nabla\W\cdot\nabla\phi
= -\int_{\Omout}\dfint\nabla w\cdot\nabla\phi\di x\,,
\end{equation}
for every $\phi\in H^1_0(\Om)$.
\end{remark}

Now, we are in the position to define the bilinear form $a:W\times W\to \R$ as
\begin{multline}\label{eq:bilin}
a\big((w,r),(\overline w,s)\big):=
\int_{\Omout}\dfint\nabla w\cdot\nabla\overline w\di x+\int_{\Omout}\dfint\nabla\W\cdot\nabla\overline w\di x
\\
+\int_{\Omout}(\dfint+\dfout)\nabla\W\cdot\nabla\overline\W\di x
+\int_{\Omout}\dfint\nabla w\cdot\nabla\overline\W\di x+\int_{\Omint}\dfdis\nabla\W\cdot\nabla\overline\W\di x
+\capdue\int_{\Memb} rs\di\sigma\,,
\end{multline}
where $\W$ and $\overline\W$ are the solutions of \eqref{eq:PDEuno}--\eqref{eq:Bound} corresponding to $(w,r)$ and
$(\overline w,s)$, respectively.

\begin{prop}\label{p:p4}
The bilinear form $a:W\times W\to \R$, defined in \eqref{eq:bilin}, is symmetric, continuous, and coercive.
\end{prop}

\begin{proof}
Notice that the bilinear form $a$ can be rewritten as
\begin{multline*}
a\big((w,r),(\overline w,s)\big)=\int_{\Omout}\dfint\nabla(w+\W)\cdot\nabla \overline w\di x+
\int_{\Omout}\dfint\nabla(w+\W)\cdot\nabla \overline\W\di x
\\
+\int_{\Omout}\dfout\nabla\W\cdot\nabla\overline\W\di x+\int_{\Omint}\dfdis\nabla\W\cdot\nabla\overline\W\di x+
\capdue \int_{\Memb}rs\di\sigma
\\
=\int_{\Omout}\dfint\nabla(w+\W)\cdot\nabla (\overline w+\overline\W)\di x+
\int_{\Omout}\dfout\nabla\W\cdot\nabla\overline\W\di x+\int_{\Omint}\dfdis\nabla\W\cdot\nabla\overline\W\di x+
\capdue \int_{\Memb}rs\di\sigma\,,
\end{multline*}
which immediately proves that it is symmetric.
Moreover, from \eqref{eq:a4}, it easily follows that $a$ is continuous.
In order to prove that it is also coercive,
we note that
\begin{multline}\label{eq:a6}
a\big((w,r),(w,r)\big)=\int_{\Omout}\dfint\nabla(w+\W)\cdot\nabla (w+\W)\di x+
\int_{\Omout}\dfout\nabla\W\cdot\nabla\W\di x
\\
+\int_{\Omint}\dfdis\nabla\W\cdot\nabla\W\di x+
\capdue \int_{\Memb}r^2\di\sigma
\\
\geq \const(\Vert \nabla w+\nabla\W\Vert^2_{L^2(\Omout)}+\Vert\nabla\W\Vert^2_{L^2(\Omout)}
+\Vert\nabla\W\Vert^2_{L^2(\Omint)} +\Vert r\Vert^2_{L^2(\Memb)})
\\
\geq \const(\Vert w\Vert^2_{H^1(\Omout)}+\Vert r\Vert^2_{H^{1/2}(\Memb)})\,,
\end{multline}
where, in the last inequality, we take into account that $r=[\W]$ (see \eqref{eq:Circuituno}) and we
use the Poincar\'{e} inequality \eqref{eq:poincare1} and the classical trace inequality, which assure that
$$
\Vert r\Vert^2_{H^{1/2}(\Memb)}\leq \const \Vert \W\Vert^2_{\X^1_0(\Omout)}\leq\const \left(\Vert \nabla\W\Vert^2_{L^2(\Omout)}
+\Vert\nabla\W\Vert^2_{L^2(\Omint)} +\Vert r\Vert^2_{L^2(\Memb)}\right)\,.
$$
\end{proof}

\begin{thm}\label{p:p2}
Assume that $\dfint,\dfout,\dfdis,\capuno,\capdue,f_1,f_2,\overline v_0, s_0$ are as in Subsection \ref{ss:position}.
Let $\widetilde u$ be the unique solution of
problem \eqref{eq:pde1}--\eqref{eq:pde3}, extended by zero inside $\Omint$, and $q$ be defined in \eqref{eq:a8}.
Then, problem \eqref{eq:PDEins}--\eqref{eq:InitData3s} admits a unique solution $(v,u)\in L^2(0,T;H^1_{null}(\Omout))\times
L^2(0,T;\X^1_0(\Om))$, such that $v\in \CC^0([0,T];L^2(\Omout))$ and $[u]\in\CC^0([0,T];L^2(\Memb))$.
\end{thm}

\begin{proof}
Let us denote by $\langle\cdot,\cdot\rangle_{W^*W}$ the duality pairing between $W$ and its dual space $W^*$ and define
$B\in W^*$ as
$$
\langle B,(\testb,[\testd])\rangle_{W^*W}=\int_{\Omout} f_1\testb\di x-\int_{\Omout}\dfint\nabla\widetilde u\cdot\nabla\testb\di x
+\int_{\Memb}q[\testd]\di\sigma\,,
$$
for every $(\testb,[\testd])\in W$.
By using the bilinear form $a$ introduced in \eqref{eq:bilin}, we can consider the following abstract problem:
\begin{equation}\label{eq:problem}
\begin{aligned}
& \qquad\hbox{find $(v,[u])\in L^2(0,T;W)\cap \CC^0([0,T];H)$}
\\
&\qquad \hbox{with $v(0)=\overline v_0$ and $[u](0)=s_0-[\widetilde u](0)$  such that}
\\
& \frac{\rm d}{{\rm d}t }\big((v,[u]),(\testb,[\testd])\big)_H+a\big((v,[u]),(\testb,[\testd])\big)
\\
&\quad = \langle B,(\testb,[\testd])\rangle_{W^*W}\,,\qquad \forall (\testb,[\testd])\in W\,,
\end{aligned}
\end{equation}
in the distributional sense. By \cite[Theorem 23.A]{Zeidler:1990}, the problem \eqref{eq:problem} is well-posed and
it is not difficult to see that its weak formulation is given by
\begin{multline}\label{eq:weak8}
-\int_{\Omout_T} v\testb\partial_t\psi\di x\di t+\int_{\Omout_T}\dfint\nabla v\cdot\nabla \testb\psi\di x\di t
+\int_{\Omout_T}\dfint\nabla u\cdot\nabla\testb\psi\di x\di t
\\
+\int_{\Omout_T}(\dfint+\dfout)\nabla u\cdot\nabla \testduno\psi \di x\di t
+\int_{\Omout_T}\dfint\nabla v\cdot\nabla \testduno\psi\di x\di t
\\
+\int_{\Omint_T}\dfdis\nabla u\cdot \nabla\testddue\psi\di x\di t
-\capuno\int_{\Memb_T}[u][\testd]\partial_t\psi\di\sigma\di t
+\capdue\int_{\Memb_T}[u][\testd]\psi\di\sigma\di t
\\
=\int_{\Omout_T}f_1\testb\psi\di x\di t
-\int_{\Omout_T}\dfint\nabla\widetilde u\cdot\nabla\testb\psi\di x\di t
+\int_{\Memb_T}q[\testd]\psi\di\sigma\di t\,,
\end{multline}
for every $(\testb,[\testd])\in W$ and every $\psi\in \CC^\infty_0(0,T)$,
where $\testd$ inside $\Omout$ and $\Omint$ is defined as the solution of
\eqref{eq:PDEuno}--\eqref{eq:Bound}, starting from $\testb\in H^1_{null}(\Omout)$ and $[\testd]\in H^{1/2}_0(\Memb,\Om)$.
Clearly, \eqref{eq:weak8} shall be complemented with the initial conditions.

Notice that \eqref{eq:weak8} formally coincides with \eqref{eq:weak1};
however, in \eqref{eq:weak1} the test function $\testd$ is a generic function belonging to $L^2(0,T;\X^1_0(\Om))$, while in the present
case it is the solution of an assigned differential problem.
Hence, in order to state that, actually, the two weak formulations are equivalent, we have to prove that we can replace
the prescribed $\testd$ in \eqref{eq:weak8} with a generic test function belonging to $\X^1_0(\Om)$.
To this purpose, let us fix $[\testd]\in H^{1/2}_0(\Memb,\Om)$
and choose two generic functions $\varphi^1\in H^1_{null}(\Omout)$ and $\varphi^2\in H^1_{null}(\Omint)$, such that
$[\varphi]:=\varphi^1-\varphi^2=[\testd]$. By \eqref{eq:weak7}, with $\phi=\testd-\varphi\in H^1_0(\Om)$, we get
\begin{multline}\label{eq:weak10}
\int_{\Omout_T}(\dfint+\dfout)\nabla u\cdot\nabla \testduno\psi \di x\di t
+\int_{\Omout_T}\dfint\nabla v\cdot\nabla \testduno\psi\di x\di t
+\int_{\Omint_T}\dfdis\nabla u\cdot \nabla\testddue\psi\di x\di t
\\
=\int_{\Omout_T}(\dfint+\dfout)\nabla u\cdot\nabla \varphi^1\psi \di x\di t
+\int_{\Omout_T}\dfint\nabla v\cdot\nabla \varphi^1\psi\di x\di t
+\int_{\Omint_T}\dfdis\nabla u\cdot \nabla\varphi^2\psi\di x\di t
\,.
\end{multline}
Hence, replacing \eqref{eq:weak10} in \eqref{eq:weak8}, it follows that it is possible to take $\testd\in \X^1_0(\Om)$
arbitrarily in \eqref{eq:weak8} and, thus, such a weak formulation coincides with \eqref{eq:weak1}, once we take into account
the density of product functions in $L^2(0,T;H^1_{null}(\Omout))$ and in $L^2(0,T;\X^1_0(\Om))$.
\end{proof}

\begin{prop}\label{p:p3}
Assume that $\dfint,\dfout,\dfdis,\capuno,\capdue,f_1,f_2,\overline v_0, s_0$ are as in Subsection \ref{ss:position}.
Assume that $\Iion\equiv 0$.
Then, problem \eqref{eq:PDEinc}--\eqref{eq:InitData3c} admits a unique solution $(V,U)\in L^2(0,T;H^1_{null}(\Omout))\times
L^2(0,T;\X^1_0(\Om))$, such that $V\in \CC^0([0,T];L^2(\Omout))$ and $[U]\in\CC^0([0,T];L^2(\Memb))$.
\end{prop}

\begin{proof}
Uniqueness easily follows by the linearity of problem \eqref{eq:PDEinc}--\eqref{eq:InitData3c}. In order to prove existence, set
$V=v$ and $U=u+\widetilde u$, where $\widetilde u\in H^1(0,T;\X^1_0(\Om))$ is the solution of problem \eqref{eq:pde1}--\eqref{eq:pde3} and
the pair $(v,u)\in L^2(0,T;H^1_{null}(\Omout))\times L^2(0,T;\X^1_0(\Om))$, with $v\in \CC^0([0,T];L^2(\Omout))$
and $[u]\in \CC^0([0,T];L^2(\Memb))$,
is the solution of \eqref{eq:PDEins}--\eqref{eq:InitData3s},
whose existence is guaranteed by Proposition \ref{p:p2}. Then, by \eqref{eq:weak1},
we get
\begin{multline}\label{eq:weak4}
-\int_{\Omout_T} V\partial_t\testb\di x\di t+\int_{\Omout_T}\dfint\nabla V\cdot\nabla \testb\di x\di t
+\int_{\Omout_T}\dfint\nabla U\cdot\nabla\testb\di x\di t
\\
+\int_{\Omout_T}(\dfint+\dfout)\nabla U\cdot\nabla \testduno \di x\di t
-\int_{\Omout_T}(\dfint+\dfout)\nabla \widetilde u\cdot\nabla \testduno \di x\di t
+\int_{\Omout_T}\dfint\nabla V\cdot\nabla \testduno\di x\di t
\\
+\int_{\Omint_T}\dfdis\nabla U\cdot \nabla\testddue\di x\di t
-\capuno\int_{\Memb_T}[U] \partial_t[\testd]\di\sigma\di t
+\capdue\int_{\Memb_T}[U][\testd]\di\sigma\di t
\\
=\int_{\Omout_T}f_1\testb\di x\di t\,,
\end{multline}
for every $\testb\in L^2(0,T;H^1_{null}(\Omout))\cap H^1_0(0,T;L^2(\Omout))$,
$\testd\in L^2(0,T;\X^1_0(\Om))$, where, as before, $[\testd]=\testduno-\testddue$ and $[\testd]\in H^1_0(0,T;L^2(\Memb))$.
Recalling \eqref{eq:weak5}, the thesis is achieved, up to an integration in time, once we have taken into account that
\begin{equation}\label{eq:a9}
\int_{\Omout_T}(\dfint+\dfout)\nabla \widetilde u\cdot\nabla \testduno \di x\di t
=\int_{\Omout_T}(f_1-f_2)\testduno\di x\di t\,,
\end{equation}
as follows from \eqref{eq:pde1}--\eqref{eq:pde3}.
\end{proof}

As a consequence of the previous results, we finally get our main theorem.

\begin{thm}\label{t:t1}
Assume that $\dfint,\dfout,\dfdis,\capuno,\capdue,f_1,f_2,\overline v_0, s_0$, and $ \Iion$ satisfy the assumptions stated
in Subsection \ref{ss:position}.
Then, problem \eqref{eq:PDEinc}--\eqref{eq:InitData3c} admits a unique solution $(V,U)\in L^2(0,T;H^1_{null}(\Omout))\times
L^2(0,T;\X^1_0(\Om))$, such that $V\in \CC^0([0,T];L^2(\Omout))$ and $[U]\in\CC^0([0,T];L^2(\Memb))$.
\end{thm}

\begin{proof}
The proof can be obtained following the same approach as in \cite{CDP2} (see, also, \cite[\S 2.4.1]{Davidovic:2016}).
Indeed, recalling that the function $g$ appearing in the gating equation \eqref{eq:gating1} is affine with respect to its second entry,
problem \eqref{eq:gating1}--\eqref{eq:gating2}
can be explicitly solved in term of $\ubuno-\ubdue=V$. Therefore, denoting by $\widetilde w_V$ such a solution and by
$h(V):=\Iion(V,\widetilde w_V)$, we obtain that problem \eqref{eq:PDEinc}--\eqref{eq:InitData3c} is a nonlinear version of the
problem considered in Proposition \ref{p:p3}. Moreover, since the nonlinearity $h$ satisfies the assumptions
\cite[Definition 4.3.1]{Cazenave:Haraux:1998},
the thesis follows by the results in \cite[Section 4.3]{Cazenave:Haraux:1998}
(see, also, \cite[Ch. 6, Theorem 1.2]{Pazy:2012}).
\end{proof}

\end{document}

%% file: fig_mb.tex
\begin{pspicture}*(12,6)
  \def\ALPHA{10 }
  \def\BETA{30 }
  \def\RADIUS{2.5 }
  \def\AMAX{3.3 }
  \psset{nameX=$x_{1}$,nameY=$x_{2}$,nameZ=$x_{3}$,Alpha=\ALPHA,Beta=\BETA,xMin=-\AMAX,xMax=\AMAX,yMin=-\AMAX,yMax=\AMAX,zMin=-0.5,zMax=\AMAX,drawing=false}
  \rput(3,3){%
    \pstThreeDCoor[linecolor=black,]
    \pstThreeDCircle[linewidth=0.4pt,beginAngle=0,endAngle=360](0,0,0)(\RADIUS \BETA sin mul \ALPHA sin mul neg,\RADIUS \BETA sin mul \ALPHA cos mul neg,\RADIUS \BETA cos mul)(\ALPHA cos,\ALPHA sin neg,0)
    \pstThreeDCircle[linewidth=0.4pt](0,0,0)(\RADIUS,0,0)(0,\RADIUS,0)
    \pstThreeDBox[hiddenLine,fillstyle=solid](-1.2,-.5,-.8)(.5,0,0)(0,.5,0)(0,0,2.3)
    \pstThreeDBox[hiddenLine](-1.2,-.5,-.8)(.5,0,0)(0,.5,0)(0,0,2.3)
    \pstThreeDBox[hiddenLine](0,-1,-.8)(1.2,0,0)(0,2,0)(0,0,.7)
    \pstThreeDBox[hiddenLine,fillstyle=solid](0,-1,.8)(1.2,0,0)(0,2,0)(0,0,.7)
    \pstThreeDBox[hiddenLine](0,-1,.8)(1.2,0,0)(0,2,0)(0,0,.7)
  }
  \rput(9,3){%
    \pstThreeDCoor[linecolor=black,]
    \pstThreeDCircle[linewidth=0.4pt,beginAngle=0,endAngle=360](0,0,0)(\RADIUS \BETA sin mul \ALPHA sin mul neg,\RADIUS \BETA sin mul \ALPHA cos mul neg,\RADIUS \BETA cos mul)(\ALPHA cos,\ALPHA sin neg,0)
    \pstThreeDCircle[linewidth=0.4pt](0,0,0)(\RADIUS,0,0)(0,\RADIUS,0)
    \pstThreeDBox[hiddenLine,fillstyle=solid](-.5,.2,-.8)(.5,0,0)(0,.5,0)(0,0,2.3)
    \pstThreeDBox[hiddenLine](-.5,.2,-.8)(.5,0,0)(0,.5,0)(0,0,2.3)
    \pstThreeDBox[hiddenLine,fillstyle=solid](0,-1,-.8)(1.2,0,0)(0,2,0)(0,0,.7)
    \pstThreeDBox[hiddenLine](0,-1,-.8)(1.2,0,0)(0,2,0)(0,0,.7)
    \pstThreeDCircle[linewidth=0.4pt,beginAngle=0,endAngle=90](0,0,0)(\RADIUS,0,0)(0,\RADIUS,0)
    \pstThreeDBox[hiddenLine,fillstyle=solid](0,-1,.8)(1.2,0,0)(0,2,0)(0,0,.7)
    \pstThreeDBox[hiddenLine](0,-1,.8)(1.2,0,0)(0,2,0)(0,0,.7)
    \pstThreeDBox[hiddenLine,linestyle=dashed](-.5,.2,-.8)(.5,0,0)(0,.5,0)(0,0,2.3)
    \pstThreeDLine[linecolor=white](0,.2,-.8)(0,.7,-.8)
    \pstThreeDLine[linecolor=white](0,.2,.8)(0,.7,.8)
    \pstThreeDSquare[linestyle=none,fillstyle=solid](0,1,.8)(.6,0,0)(0,-.8,0)
    \pstThreeDLine(1.2,1,.8)(0,1,.8)
    \pstThreeDSquare[fillstyle=solid,linestyle=none](.9,-.5,1.5)(-.6,0,0)(0,0,.5)
    \pstThreeDSquare[fillstyle=solid,linestyle=none](.3,-.5,1.5)(.6,0,0)(0,1,0)
    \pstThreeDLine(.3,-.5,1.5)(.9,-.5,1.5)(.9,.5,1.5)(.3,.5,1.5)
    \pstThreeDLine[linestyle=dashed](.3,.5,1.5)(.3,-.5,1.5)
    \pstThreeDLine[linestyle=dashed](1.2,-1,.8)(0,-1,.8)
    \pstThreeDLine(.9,-.5,1.5)(.9,-.5,1.71464)
    \pstThreeDLine(.3,-.5,1.5)(.3,-.5,1.91311)
    \pstThreeDLine(.9,.5,1.5)(.9,.5,1.71464)
    \pstThreeDLine(.3,.5,1.5)(.3,.5,1.91311)
    \parametricplotThreeD[linewidth=.4 pt](.3,.9){t .5 neg 4 .25 sub t 2 exp sub .5 exp}
    \parametricplotThreeD[linewidth=.4 pt](-.5,.5){.3 t 4 .09 sub t 2 exp sub .5 exp}
    \parametricplotThreeD[linewidth=.4 pt](.3,.9){t .5 4 .25 sub t 2 exp sub .5 exp} 
    \parametricplotThreeD[linewidth=.4 pt](-.5,.5){.9 t 4 .81 sub t 2 exp sub .5 exp}
    \pstThreeDLine[linestyle=dashed](.8,-1,1.5)(0,-1,1.5)
    \pstThreeDLine(1.2,-1,1.5)(.8,-1,1.5)
  }

\end{pspicture}
